\documentclass[11pt]{article}
\usepackage{amsmath,amsfonts,amssymb,amsthm}
\usepackage{geometry}
\usepackage{enumitem}
\usepackage{fancyhdr}
\usepackage{bbm}
\usepackage{hyperref}
\hypersetup{colorlinks=false}
\usepackage[
backend=biber,
style=numeric,
sorting=nyt
]{biblatex}

\allowdisplaybreaks
 
\newtheorem{theorem}{Theorem}[section]
\newtheorem{corollary}[theorem]{Corollary}
\newtheorem{lemma}[theorem]{Lemma}
\newtheorem{proposition}[theorem]{Proposition}
\newtheorem{conjecture}[theorem]{Conjecture}

\theoremstyle{definition}
\newtheorem{definition}[theorem]{Definition}
\newtheorem{example}[theorem]{Example}
\newtheorem{remark}[theorem]{Remark}

\renewcommand{\k}{\mathbbm{k}}

\setlist[enumerate]{label=\arabic*),ref=\arabic*)}

\title{Toward explicit Hilbert series of quasi-invariant polynomials in characteristic $p$ and $q$-deformed quasi-invariants}
\author{Frank Wang}
\date{\vspace{-5ex}}

\newcommand{\triv}{\mathrm{triv}}
\newcommand{\std}{\mathrm{std}}
\newcommand{\sgn}{\mathrm{sign}}
\newcommand{\Irrep}{\mathrm{Irrep}}
\renewcommand{\char}{\text{char\,}}
\newcommand{\C}{\mathbb{C}}
\renewcommand{\b}[1]{\textbf{#1}}
\newcommand{\Z}{\mathbb{Z}}

\begin{document}

\maketitle
\begin{abstract}
    We study the spaces $Q_m$ of $m$-quasi-invariant polynomials of the symmetric group $S_n$ in characteristic $p$. Using the representation theory of the symmetric group we describe the Hilbert series of $Q_m$ for $n=3$, proving a conjecture of Ren and Xu \cite{Ren_2020}. From this we may deduce the palindromicity and highest term of the Hilbert polynomial and the freeness of $Q_m$ as a module over the ring of symmetric polynomials, which are conjectured for general $n$. We also prove further results in the case $n=3$ that allow us to compute values of $m,p$ for which $Q_m$ has a different Hilbert series over characteristic 0 and characteristic $p$, and what the degrees of the generators of $Q_m$ are in such cases. We also extend various results to the spaces $Q_{m,q}$ of $q$-deformed $m$-quasi-invariants and prove a sufficient condition for the Hilbert series of $Q_{m,q}$ to differ from the Hilbert series of $Q_m$.
\end{abstract}

\section{Introduction}
Let $\k$ be a field, and let $S_n$ act on $V=\k^n$ by permuting the coordinates. A commonly studied object is the space of \textit{invariant} or \textit{symmetric} polynomials, which can be described as the set of $P\in\k[V]=\k[x_1,\dots,x_n]$ such that for all transpositions $s_{i,j}\in S_n$, $(1-s_{i,j})P=0$. A generalization of the space of invariant polynomials is the space $Q_m$ of $m$-quasi-invariant polynomials, defined as the set of $P\in\k[V]$ such that for all transpositions $s_{i,j}\in S_n$, $(1-s_{i,j})P$ is divisible by $(x_i-x_j)^{2m+1}$. Thus $m$-quasi-invariant polynomials are polynomials such that $(1-s_{i,j})P$ vanishes on the reflection hyperplane $x_i=x_j$ to some high order, a generalization of invariant polynomials, where $(1-s_{i,j})P$ vanishes on this hyperplane to infinite order.

Quasi-invariants were introduced by Chalykh and Veselov \cite{cmp/1104179957} in their work on quantum Calogero-Moser systems. Quasi-invariant polynomials were also found to be useful in the study of rational Cherednik algebras, where they can be used to describe representations of the spherical subalgebra of a rational Cherednik algebra \cite{berest2003cherednik}. These applications in the theory of integrable systems and representation theory have generated interest in the subject, leading to many interesting developments. In \cite{feigin}, Feigin and Veselov studied the relation between quasi-invariants and Calogero-Moser systems. In \cite{felder2003action}, Felder and Veselov computed the Hilbert series of $Q_m$ in characteristic zero using the Knizhnik-Zamolodchikov equations. Berest and Chalykh later defined quasi-invariant polynomials of a general complex reflection group \cite{berest_chalykh_2011}. Quasi-invariants in characteristic $p$ were first studied by Ren and Xu, where they computed the Hilbert series of $Q_m$ for $S_2$ and proved sufficient conditions for the Hilbert series of $Q_m$ in characteristic $p$ to differ from the Hilbert series in characteristic zero \cite{Ren_2020}.

The $q$-deformed quasi-invariants are a natural deformation of quasi-invariants introduced by Chalykh \cite{Chalykh2002MacdonaldPA} to prove various conjectures about Macdonald polynomials. In \cite{BEF}, Braverman, Etingof, and Finkelberg introduced a deformation of the rational Cherednik algebra which they called the cyclotomic DAHA. They showed that the cyclotomic DAHA acts on the space $Q_{m,q}$ of $q$-deformed quasi-invariants and used this to show that $Q_{m,q}$ is a flat deformation of $Q_m$ for all but countably many values of $q$. In this paper, we work toward a characterization of the specific values of $q$ for which this deformation is not flat, and what the Hilbert series of $Q_{m,q}$ is in such cases.

In Section \ref{sec:2} we give basic definitions and properties of $Q_m$ for an arbitrary symmetric group. In Section \ref{sec:3} we focus on the case $S_3$ and prove our main result about the Hilbert series of $Q_m$. In Section \ref{sec:4} we discuss how our results in Section \ref{sec:3} might be used to determine the values of $m,p$ such that the Hilbert series of $Q_m$ in characteristic $p$ differs from that in characteristic zero and to compute the Hilbert series explicitly in such cases. We also describe some partial progress we made in this direction using the Opdam shift operator. In Section \ref{sec:5} we discuss some preliminary results and conjectures about the Hilbert series of $q$-deformed quasi-invariants, particularly in the case $S_3$.

\section{Quasi-invariant polynomials for $S_n$}\label{sec:2}

We start with the definition of quasi-invariant polynomials. Let $\k$ be a field, and let $S_n$ act on $\k[x_1,\dots,x_n]$ by permuting the variables. Let $s_{i,j}\in S_n$ denote the transposition swapping the $i$-th and $j$-th indices.

\begin{definition}
    Let $n$ be a positive integer. A polynomial $P\in\k[x_1,\dots,x_n]$ is $m$-quasi-invariant if for all transpositions $s_{i,j}\in S_n$, $(1-s_{i,j})P$ is divisible by $(x_i-x_j)^{2m+1}$. We denote by $Q_m(n,\k)$ the space of all $m$-quasi-invariants in $\k[x_1,\dots,x_n]$. 
\end{definition}

Let $\k[x_1,\dots,x_n]^{S_n}$ denote the space of polynomials in $\k[x_1,\dots,x_n]$ which are invariant with respect to the action of $S_n$. Some useful properties of $Q_m(n,\k)$ are described in the following proposition. 

\begin{proposition}[\cite{etingof2002lectures}]\label{prop:2.2}
    \phantom{.}

    \begin{enumerate}
        \item $\k[x_1,\dots,x_n]^{S_n}\subset Q_m(n,\k)$, $Q_0(n,\k)=\k[x_1,\dots,x_n]$, $Q_m(n,\k)\subset Q_{m'}(n,\k)$ if $m>m'$.
        \item $Q_m(n,\k)$ is a ring.
        \item $Q_m(n,\k)$ is a finitely generated module over $\k[x_1,\dots,x_n]^{S_n}$.
    \end{enumerate}
\end{proposition}

Note that $Q_m(n,\k)$ admits a grading by the degrees of the polynomials. Thus we may define its Hilbert series.

\begin{definition}
    The Hilbert series of $Q_m(n,\k)$ is
    $$H_m(t)=\sum_{d\geq 0}\dim Q_m(n,\k)[d]t^d$$
    where $Q_m(n,\k)[d]$ is the graded component of $Q_m(n,\k)$ consisting of polynomials of homogeneous degree $d$.
\end{definition}

Note that by the fundamental theorem of symmetric polynomials, $\k[x_1,\dots,x_n]^{S_n}$ is a free polynomial algebra generated by symmetric polynomials $P_1,\dots,P_n$ of degrees $1,\dots,n$, respectively \cite{humphreys_1990}. Thus since $Q_m(n,\k)$ is a finitely generated module over $\k[x_1,\dots,x_n]^{S_n}=\k[P_1,\dots,P_n]$ by Proposition \ref{prop:2.2}, we can write
$$H_m(t)=\frac{G_m(t)}{\prod_{d=1}^n(1-t^d)}$$
where $G_m(t)$ is a polynomial by Hilbert's syzygy theorem. We say that $G_m(t)$ is the Hilbert polynomial associated with $H_m(t)$. Note that describing the Hilbert polynomial is related to describing the generators and relations of $Q_m(n,\k)$ as a $\k[x_1,\dots,x_n]^{S_n}$-module.

From now on, we will assume the characteristic of $\k$ does not divide $n!$, so that the representation theory of $S_n$ over $\k$ is non-modular. The following lemma and proposition are standard results in the study of quasi-invariants, for example they a direct consequence of the decomposition of $Q_m(n,\C)$ as a representation of the spherical rational Cherednik algebra (see \cite{berest2003cherednik}) in the case $\k=\C$. 

\begin{lemma}\label{lem:2.4}
    $Q_m(n,\k)$ is a representation of $S_n$.
\end{lemma}

Now, consider a graded component $Q_m(n,\k)[d]$ of the quasi-invariants. It is a subrepresentation of $Q_m(n,\k)$, and moreover it is finite dimensional since it is a subspace of the finite dimensional space $\k[x_1,\dots,x_n][d]$. Thus, since we are working in the non-modular case, by Maschke's theorem it decomposes as a direct sum
$$Q_m(n,\k)[d]=\bigoplus_{\tau\in\Irrep(S_n)}Q_m(n,\k)[d]_\tau$$
where $\Irrep(S_n)$ is the set of irreducible representations of $S_n$ and $Q_m(n,\k)[d]_\tau$ is the subspace of $Q_m(n,\k)[d]$ on which $S_n$ acts by a direct sum of copies of $\tau$. Then for $\tau\in\Irrep(S_n)$, let us define
$$Q_m(n,\k)_\tau=\bigoplus_{d\geq 0}Q_m(n,\k)[d]_\tau$$
to be the isotypic components of $Q_m(n,\k)$.

\begin{proposition}\label{prop:2.5}
\begin{enumerate}
    \item $Q_m(n,\k)_\tau$ is a module over $\k[x_1,\dots,x_n]^{S_n}$.\label{prop:2.51}
    \item We have the decomposition
    $$Q_m(n,\k)=\bigoplus_{\tau\in\Irrep(S_n)}Q_m(n,\k)_\tau$$
    as a $\k$-vector space and hence as a $\k[x_1,\dots,x_n]^{S_n}$-module as well.\label{prop:2.52}
\end{enumerate}
\end{proposition}

By the above proposition, to study the generators and relations of $Q_m(n,\k)$ it suffices to study the generators and relations of $Q_m(n,\k)_\tau$. This will be our primary method of studying $Q_m(n,\k)$ throughout the rest of this paper. Let us denote by triv and sign the trivial and sign representations of $S_n$, respectively. We deal with these two cases here.

\begin{proposition}\label{prop:2.6} Assume $\char\k\neq 2$. As $\k[x_1,\dots,x_n]^{S_n}$-modules, we have
\begin{enumerate}
    \item $Q_m(n,\k)_\triv$ is freely generated by $1$.\label{prop:2.61}
    \item $Q_m(n,\k)_\sgn$ is freely generated by $\prod_{i<j}(x_i-x_j)^{2m+1}$.\label{prop:2.62}
\end{enumerate}
\end{proposition}
\begin{proof}
    \ref{prop:2.61}~Note that the statement that $S_n$ acts by the trivial representation on $P$ is equivalent to the statement that $P$ is $S_n$-invariant, so $Q_m(n,\k)_\triv\subset\k[x_1,\dots,x_n]^{S_n}$. On the other hand, for any $P\in\k[x_1,\dots,x_n]^{S_n}$ and any $s_{i,j}\in S_n$ we have $(1-s_{i,j})P=P-P=0$ is divisible by $(x_i-x_j)^{2m+1}$, so $Q_m(n,\k)_\triv=\k[x_1,\dots,x_n]^{S_n}$.
    
    \ref{prop:2.62}~Let $P\in Q_m(n,\k)_\sgn$ and let $s_{i,j}\in S_n$. Then since $P$ is in the sign representation we have $s_{i,j}P=-P$, so $(1-s_{i,j})P=2P$ is divisible by $(x_i-x_j)^{2m+1}$. So $P$ is divisible by $(x_i-x_j)^{2m+1}$ for all $i$ and $j$, and thus it is divisible by $\prod_{i<j}(x_i-x_j)^{2m+1}$. 
    
    Let us write $P=K\prod_{i<j}(x_i-x_j)^{2m+1}$ for some polynomial $K$. Note that 
    $$s_{i',j'}\prod_{i<j}(x_i-x_j)^{2m+1}=-\prod_{i<j}(x_i-x_j)^{2m+1}$$
    for all $s_{i',j'}\in S_n$, so we have 
    $$s_{i',j'}P=-P=-K\prod_{i<j}(x_i-x_j)^{2m+1}=s_{i',j'}K\bigg(s_{i',j'}\prod_{i<j}(x_i-x_j)^{2m+1}\bigg)=-s_{i',j'}K\prod_{i<j}(x_i-x_j)^{2m+1}$$
    so $s_{i',j'}K=K$. Since $S_n$ is generated by transpositions, $K$ is $S_n$-invariant, and thus $P$ is in the $\k[x_1,\dots,x_n]^{S_n}$-module generated by $\prod_{i<j}(x_i-x_j)^{2m+1}$. Conversely, for any $K\in\k[x_1,\dots,x_n]^{S_n}$ and $s_{i',j'}\in S_n$, we have 
    $$s_{i',j'}\Big(K\prod_{i<j}(x_i-x_j)^{2m+1}\Big)=-K\prod_{i<j}(x_i-x_j)^{2m+1}$$
    and $$(1-s_{i',j'})\Big(K\prod_{i<j}(x_i-x_j)^{2m+1}\Big)=2K\prod_{i<j}(x_i-x_j)^{2m+1}$$ is divisible by $(x_{i'}-x_{j'})^{2m+1}$, so the $\k[x_1,\dots,x_n]^{S_n}$-module generated by $\prod_{i<j}(x_i-x_j)^{2m+1}$ is contained in $Q_m(n,\k)_\sgn$. So $Q_m(n,\k)_\sgn$ is exactly this module. Freeness then follows from the fact that $\k[x_1,\dots,x_n]$ is an integral domain.
\end{proof}

\begin{remark}\label{rem:2.7}
    It was proved in \cite{Ren_2020} that the Hilbert series of $Q_m(2,\k)$ is $\frac{1+t^{2m+1}}{(1-t)(1-t^2)}$, regardless of the characteristic of $\k$. Proposition \ref{prop:2.6} provides an alternate proof of this result in the case $\char \k\neq 2$. Namely, note that the only irreducible representations of $S_2$ are triv and sign. Thus the generators of $Q_m(2,\k)$ are 1 (of degree 0) and $(x_1-x_2)^{2m+1}$ (of degree $2m+1$). Since these generators belong to different irreducible representations of $S_2$, there are no relations between them, and thus the Hilbert polynomial is $1+t^{2m+1}$ and the Hilbert series is $\frac{1+t^{2m+1}}{(1-t)(1-t^2)}$.
\end{remark}

\section{The Hilbert series of $Q_m(3,\k)$}\label{sec:3}

The Hilbert series of $Q_m(3,\C)$ was computed in \cite{felder2003action} and is as follows:
\begin{equation}\label{eq:3.1}
    H_m(t)=\frac{1+2t^{3m+1}+2t^{3m+2}+t^{6m+3}}{(1-t)(1-t^2)(1-t^3)}.\tag{3.1}
\end{equation}

In this section, we work toward an explicit description of the Hilbert series of $Q_m(3,\k)$ in characteristic $p>3$. Note that there are three irreducible representations of $S_3$: triv, sign, and the 2-dimensional standard representation which we will denote by std. Proposition \ref{prop:2.6} computes the generators of $Q_m(3,\k)_\triv$ and $Q_m(3,\k)_\sgn$, so to compute the Hilbert series of $Q_m(3,\k)$ we only need to consider the space $Q_m(3,\k)_\std$. Both this section and the next section will be largely focused on understanding this space.

This section is dedicated to proving the main result of this paper, which is the following theorem, conjectured in \cite{Ren_2020}. Note the similarity to (\ref{eq:3.1}).

\begin{theorem}\label{thm:3.1}
    Let $\k$ be a field of characteristic $p>3$. Then $Q_m(3,\k)_\std$ is freely generated by standard representations of degree $d$ and $6m+3-d$ where $d$ is an integer depending on $m,p$ that satisfies $2m+1\leq d\leq 3m+1$. In particular, the Hilbert series of $Q_m(3,\k)$ is
    $$H_m(t)=\frac{1+2t^d+2t^{6m+3-d}+t^{6m+3}}{(1-t)(1-t^2)(1-t^3)}.$$
    
\end{theorem}

To start, let $V$ be a copy of the standard representation, and let $s_{i,j}\in S_3$. Then $V$ contains a one dimensional $-1$ eigenspace of $s_{i,j}$. Let us denote this eigenspace by $V_{i,j}^-$. Then we have

\begin{lemma}\label{lem:3.2}
    Let $V\subset Q_m(3,\k)_\std$ be a copy of the standard representation, and let $P\in V_{i,j}^-$. Then we have $P+sP+s^2P$=0 where $s=(1\,2\,3)\in S_3$ and $P=(x_i-x_j)^{2m+1}K$ for some polynomial $K$ that is invariant under the action of $s_{i,j}$. Conversely, let $K$ be an $s_{1,2}$-invariant polynomial such that 
    $$(x_1-x_2)^{2m+1}K+(x_2-x_3)^{2m+1}sK+(x_3-x_1)^{2m+1}s^2K=0.$$
    Then $(x_1-x_2)^{2m+1}K$ either belongs to $Q_m(3,\k)_\sgn$ or the $-1$ eigenspace of $s_{1,2}$ in some copy of std inside $Q_m(3,\k)_\std$.
\end{lemma}
\begin{proof}
    The proof of the first statement is analogous to the proof of Proposition \ref{prop:2.6}. Since $P\in V_{i,j}^-$, we have $(1-s_{i,j})P=2P$ is divisible by $(x_i-x_j)^{2m+1}$, so $P=(x_i-x_j)^{2m+1}K$ for some polynomial $K$. Then since $P$ and $(x_i-x_j)^{2m+1}$ are both $s_{i,j}$-antiinvariant, $K$ must be $s_{i,j}$-invariant. The fact that $P+sP+s^2P=0$ is a direct consequence of $P$ being an element of the standard representation.
    
    For the second statement, note that since $K$ is $s_{1,2}$-invariant, $sK=s_{1,3}K$ and $s^2K=s_{2,3}K$. So we have
    $$(1-s_{1,3})\Big((x_1-x_2)^{2m+1}K\Big)=(x_1-x_2)^{2m+1}K-(x_3-x_2)^{2m+1}sK=-(x_3-x_1)^{2m+1}s^2K,$$
    $$(1-s_{2,3})\Big((x_1-x_2)^{2m+1}K\Big)=(x_1-x_2)^{2m+1}K-(x_1-x_3)^{2m+1}s^2K=-(x_2-x_3)^{2m+1}sK$$
    so $(x_1-x_2)^{2m+1}K$ is $m$-quasi-invariant. On the other hand, since $s_{1,2}$ sends the vector space spanned by $(x_1-x_2)^{2m+1}K$ to itself and $s_{1,3}$ and $s_{2,3}$ send $(x_1-x_2)^{2m+1}K$ to polynomials in the space spanned by $(x_1-x_2)^{2m+1}K$ and $(x_2-x_3)^{2m+1}sK$, the span of the $S_3$ orbit of $(x_1-x_2)^{2m+1}K$ is at most 2 dimensional. If it is 1 dimensional, then $(x_1-x_2)^{2m+1}K$ is contained in a copy of sign inside $Q_m(3,\k)_\sgn$ since $s_{1,2}$ acts by $-1$. If it is 2 dimensional, then since the only 2 dimensional irreducible representation of $S_3$ is std, $(x_1-x_2)^{2m+1}K$ belongs to some copy of std inside $Q_m(3,\k)_\std$. That it is in the $-1$ eigenspace of $s_{1,2}$ is clear by applying $s_{1,2}$ to $(x_1-x_2)^{2m+1}K$.
\end{proof}

Note that since $Q_m(3,\k)_\std$ is a graded $\k[x_1,x_2,x_3]^{S_3}$-module, we may assume its generators are homogeneous. We may also describe generating representations of $Q_m(3,\k)_\std$, rather than generating elements. 

\begin{corollary}\label{cor:3.3}
    Let $V$ be a generating representation of $Q_m(3,\k)_\std$ and let $P\in V_{i,j}^-$. Let us write $P=(x_i-x_j)^{2m+1}K$. Then $K$ is not divisible by any nonconstant symmetric polynomial.
\end{corollary}
\begin{proof}
    Since $P$ is in $Q_m(3,\k)_\std$, it satisfies $P+sP+s^2P=0$ where $s=(1\,2\,3)$. Assume for contradiction that $K$ is divisible by some symmetric polynomial $Q$. Then it is clear that \linebreak
    $P/Q+s(P/Q)+s^2(P/Q)=0$ and $P/Q=(x_i-x_j)^{2m+1}(K/Q)$, so $P/Q$ is contained in $Q_m(3,\k)$ by Lemma \ref{lem:3.2}. But then $P$ is in the $\k[x_1,x_2,x_3]^{S_3}$-module generated by $P/Q$, so it is generated by an element of lower degree and thus is not contained in a generating representation.
\end{proof}

For a standard representation $V$ consisting of homogeneous polynomials of degree $d$, let us say its degree is $\deg V=d$. The following lemma makes use of this, and will prove to be useful both in the proof of Theorem \ref{thm:3.1} and in Section \ref{sec:4}.

\begin{lemma}\label{lem:3.4}
    Let $V,W$ be distinct generating representations of $Q_m(3,\k)_\std$. Let $A\in V_{1,2}^-,B\in W_{1,2}^-$. Then $As_{2,3}B-Bs_{2,3}A$ is a nonzero element of $Q_m(3,\k)_\sgn$ and we have $\deg V+\deg W\geq 6m+3$.
\end{lemma}
\begin{proof}
    $As_{2,3}B-Bs_{2,3}A$ is contained in $Q_m(3,\k)$ because the quasi-invariants form a ring by Proposition \ref{prop:2.2}. It is also contained in a sign representation since it spans the representation $\wedge^2\std=\sgn$ inside $V\otimes W$. Thus it is contained in $Q_m(3,\k)_\sgn$, so it remains to show that it is nonzero.
    
    By Lemma \ref{lem:3.2} we can write $A=(x_1-x_2)^{2m+1}K,B=(x_1-x_2)^{2m+1}L$ for $s_{1,2}$-invariant polynomials $K$ and $L$. Then we have
    \begin{align*}
        As_{2,3}B-Bs_{2,3}A&=(x_1-x_2)^{2m+1}K(x_1-x_3)^{2m+1}s_{2,3}L-(x_1-x_2)^{2m+1}L(x_1-x_3)^{2m+1}s_{2,3}K\\
        &=(x_1-x_2)^{2m+1}(x_1-x_3)^{2m+1}(Ks_{2,3}L-Ls_{2,3}K).
    \end{align*}
    So $As_{2,3}B-Bs_{2,3}A$ is zero if and only if $Ks_{2,3}L=Ls_{2,3}K$, or in other words if $Ks_{2,3}L$ is $s_{2,3}$-invariant. Let us assume that $Ks_{2,3}L$ is $s_{2,3}$-invariant. We will show that this implies $K=L$ up to a scalar, which contradicts $V$ and $W$ being distinct.
    
    Decompose $K$ into its irreducible factors, and let $Q$ be such a factor. We will show that $Q$ is also a factor of $L$. Since $Ks_{2,3}L$ is $s_{2,3}$-invariant, $s_{2,3}Q$ must be a factor of either $K$ or $s_{2,3}L$. If it is a factor of $s_{2,3}L$, then $P$ is a factor of $L$, as desired. Otherwise, $s_{2,3}Q$ is a factor of $K$. Then we also have that $s_{1,2}Q$ and $s_{1,2}s_{2,3}Q$ are factors of $K$ since $K$ is $s_{1,2}$-invariant. Then $s_{2,3}s_{1,2}Q$ is a factor of $Ks_{2,3}L$. If it is a factor in $s_{2,3}L$, then note that $s_{1,3}s_{2,3}s_{1,2}Q$ is a factor of $s_{2,3}L$ since $L$ being $s_{1,2}$-invariant implies $s_{2,3}L$ is $s_{1,3}$-invariant. But we have $s_{1,3}s_{2,3}s_{1,2}=s_{2,3}$, so $s_{2,3}Q$ is a factor of $s_{2,3}L$, which implies $Q$ is a factor of $L$, as desired. So we may assume $s_{2,3}s_{1,2}Q$ is a factor of $K$. Then we also have $s_{1,2}s_{2,3}s_{1,2}Q=s_{1,3}Q$ is a factor of $K$ since $K$ is $s_{1,2}$-invariant, and thus $K$ is divisible by $sQ$ for all $s\in S_3$. 
    
    So $K$ is divisible by the least common multiple of all of the $sQ$, which we will denote by $S$. $S_3$ fixes $S$ as a factor of $K$, so it must act on $S$ by scalars. It cannot act on $S$ trivially, as then $S$ would be a symmetric polynomial factor of $K$, and $K$ has no symmetric polynomial factors by Corollary \ref{cor:3.3}. So $S_3$ must act on $S$ by sign. Note also that we can only have one such factor of $K$ on which $S_3$ acts by sign, as otherwise their product would be a symmetric polynomial factor of $K$. So every factor of $K$ is also a factor of $L$ except for possibly a factor in the sign representation. We may apply everything above to $L$ as well to get the same result in the other direction, yielding a bijection between all irreducible factors of $K$ and $L$ except maybe a factor of the sign representation in either case. Since $K$ and $L$ are both $s_{1,2}$-invariant, acting on both of them by $s_{1,2}$ makes it clear that either they both have an extra factor of the sign representation or they both do not.
    
    If $K$ and $L$ do not have an additional factor in the sign representation, then we have a bijection between all irreducible factors of $K$ and $L$, and thus they are equal up to a scalar, as desired. Otherwise, note that by Proposition \ref{prop:2.2}, the sign factors lie in $Q_0(3,\k)$, and by Proposition \ref{prop:2.6} they are symmetric polynomial multiples of $(x_1-x_2)(x_2-x_3)(x_3-x_1)$. But these symmetric polynomial multiples must be 1, as otherwise $K$ and $L$ would be divisible by nonconstant symmetric polynomials, which violates Corollary \ref{cor:3.3}. So the sign factors are both $(x_1-x_2)(x_2-x_3)(x_3-x_1)$. Thus we still have a bijection between the factors of $K$ and $L$, so $K$ and $L$ are equal up to a scalar, as desired.
    
    Note that $As_{2,3}B-Bs_{2,3}A$ has degree $\deg V+\deg W$. Since $Q_m(3,\k)_\sgn$ is generated by a polynomial of degree $6m+3$ by Proposition \ref{prop:2.6}, we have $\deg V+\deg W\geq 6m+3$.
\end{proof}

The above lemma brings us closer to the proof of Theorem \ref{thm:3.1}, however we still need a couple more lemmas.

\begin{lemma}\label{lem:3.5}
    Assume that there exist generating representations $V,W$ of $Q_m(3,\k)_\std$ such that $\deg V+\deg W=6m+3$. Then $Q_m(3,\k)_\std$ is a free module over $\k[x_1,x_2,x_3]^{S_3}$ generated by $V$ and $W$.
\end{lemma}
\begin{proof}
    Assume for contradiction that there exists some other generating representation $U$ of \linebreak
    $Q_m(3,\k)_\std$. By Lemma \ref{lem:3.4} we have $\deg U\geq \deg W$. Let $A\in V_{1,2}^-,B\in W_{1,2}^-,C\in U_{1,2}^-$. Then note that by Lemma \ref{lem:3.4} we have
    $$As_{2,3}B-Bs_{2,3}A=c\prod_{i<j}(x_i-x_j)^{2m+1}$$ 
    where $c$ is a nonzero constant, since it is a degree $6m+3$ polynomial in $Q_m(3,\k)_\sgn$. We also have
    $$As_{2,3}C-Cs_{2,3}A=Q\prod_{i<j}(x_i-x_j)^{2m+1}$$
    for some symmetric polynomial $Q$. But then instead of $U$, we may take the representation generated by $cC-QB$ as the generating representation of $Q_m(3,\k)_\std$. But we have
    \begin{align*}
        As_{2,3}(cC-QB)-(cC-QB)s_{2,3}A&=c(As_{2,3}C-Cs_{2,3}A)-Q(As_{2,3}B-Bs_{2,3}A)\\
        &=(cQ-Qc)\prod_{i<j}(x_i-x_j)^{2m+1}\\
        &=0
    \end{align*}
    which contradicts Lemma \ref{lem:3.4}. If $Q_m(3,\k)_\std$ was not a free module then there would be nonzero symmetric polynomials $R_1,R_2$ such that $R_1A=R_2B$. But then we would have 
    $$R_1As_{2,3}(R_2B)-R_2Bs_{2,3}(R_1A)=R_1As_{2,3}(R_1A)-R_1As_{2,3}(R_1A)=0$$
    and 
    $$R_1As_{2,3}(R_2B)-R_2Bs_{2,3}(R_1A)=R_1R_2(As_{2,3}B-Bs_{2,3}A)$$
    which cannot happen since none of $R_1,R_2,as_{2,3}b-bs_{2,3}a$ are zero.
\end{proof}

Now, note that we have $\k[x_1,x_2,x_3]=\k[x_1+x_2+x_3,x_1-x_3,x_2-x_3]$, so for any $P\in Q_m(3,\k)$ we may write it as $P=P'+P''$, where $P'$ is divisible by $x_1+x_2+x_3$ and $P''\in\k[x_1-x_3,x_2-x_3]$. Since $x_1+x_2+x_3$ is symmetric, the action of $S_3$ preserves this decomposition, so for any $s_{i,j}\in S_3$ we have $(1-s_{i,j})P=(1-s_{i,j})P'+(1-s_{i,j})P''$. Then since $(x_i-x_j)^{2m+1}\in\k[x_1-x_3,x_2-x_3]$, it follows that $P',P''\in Q_m(3,\k)$. Moreover, if $P$ is a generator of $Q_m(3,\k)$, then $P''$ is also a generator since $P'$ is generated by $\frac{P'}{x_1+x_2+x_3}$. So we may assume all generators of $Q_m(3,\k)$ belong to $\k[x_1-x_3,x_2-x_3]$. 

\begin{lemma}\label{lem:3.6}
    Let $V,W$ be generating representations of $Q_m(3,\k)_\std$ such that $\deg V<\deg W$ and $\deg W-\deg V>1$. Then we may take the generating representations to be $V,W'$ such that either $V$ or $W'$ is in $Q_{m+1}(3,\k)_\std$.
\end{lemma}
\begin{proof}
    Let $A\in V_{1,2}^-,B\in V_{1,2}^-$ and let $A=(x_1-x_2)^{2m+1}K,B=(x_1-x_2)^{2m+1}L$. We may assume $K,L\in\k[x_1-x_3,x_2-x_3]$. Note that since $K,L$ are invariant with respect to $s_{1,2}$, we have that $K,L$ are symmetric in $x_1-x_3$ and $x_2-x_3$. Recall from Section \ref{sec:2} that by the fundamental theorem of symmetric polynomials, $\k[x_1,x_2,x_3]^{S_n}$ as an algebra is generated by polynomials $P_1,P_2,P_3$ of degrees $1,2,3$, respectively. We may take 
    $$P_2=(x_1-x_2)^2+(x_2-x_3)^2+(x_3-x_1)^2,\quad P_3=(x_1+x_2-2x_3)(x_2+x_3-2x_1)(x_3+x_1-2x_2)$$
    so that $P_2,P_3\in\k[x_1-x_3,x_2-x_3]$. Because $\deg L-\deg K=\deg W-\deg V>1$, there exists some nonzero symmetric polynomial $P$ generated by $P_2,P_3$ of degree $\deg L-\deg K$. Moreover, note that 
    $$P_2^3-2P_3^2=54\prod_{i<j}(x_i-x_j)^2$$
    so if we choose $Q$ to be a monomial in $P_2,P_3$ then it will not be divisible by $\prod_{i<j}(x_i-x_j)^2$ since there are no polynomial relations between $P_2,P_3$. 
    Let us fix such a polynomial $Q$.
    
    Note that $A,B\in Q_{m+1}(3,\k)$ if and only if $K,L$ are divisible by $(x_1-x_2)^2\in \k[x_1-x_3,x_2-x_3]$, respectively. Then consider the images of $QK$ and $L$ in $\k[x_1-x_3,x_2-x_3]/((x_1-x_2)^2)$. First, consider $QK$. We have $Q\neq 0$, since $Q$ is not divisible by $\prod_{i<j}(x_i-x_j)^2$ and any symmetric polynomial divisible by $(x_1-x_2)^2$ must also be divisible by $\prod_{i<j}(x_i-x_j)^2$. If $K=0$ then $K$ is divisible by $(x_1-x_2)^2$, so $K\in Q_{m+1}(3,\k)$ and we may take $W'=W$ to complete the proof. So we may assume $K\neq 0$. Lastly, note that the only zero divisors in $\k[x_1-x_3,x_2-x_3]/((x_1-x_2)^2)$ are multiples of $x_1-x_2$. But since $Q,K$ are both $s_{1,2}$-invariant, if either one is divisible by $x_1-x_2$ then it is also divisible by $(x_1-x_2)^2$. But we already dealt with this case for both $Q$ and $K$, so we may assume that neither $Q$ nor $K$ are zero divisors. Thus $QK\neq 0$. 
    
    Now, note that when working in $\k[x_1-x_3,x_2-x_3]/((x_1-x_2)^2)$, from any homogeneous polynomial we can cancel out terms where the difference in the degrees of $x_1-x_3$ and $x_2-x_3$ is large. Since $QK$ and $L$ are symmetric in $x_1-x_3$ and $x_2-x_3$, we can cancel out terms until we eventually end up with either
    $$QK=c_K(x_1-x_3)^d(x_2-x_3)^d,\quad L=c_L(x_1-x_3)^d(x_2-x_3)^d$$
    or 
    $$QK=c_K(x_1+x_2-2x_3)(x_1-x_3)^d(x_2-x_3)^d,\quad L=c_L(x_1+x_2-2x_3)(x_1-x_3)^d(x_2-x_3)^d$$
    for scalars $c_K,c_L$ and $d=\Big\lfloor\frac{\deg L}{2}\Big\rfloor$ depending on whether $\deg L$ is even or odd. Note that $c_K\neq 0$, as otherwise we would have $QK=0$. Then note that $L-\frac{c_L}{c_K}QK=0$. Set $L'=L-\frac{c_L}{c_K}QK$. Then $L'$ is in $Q_{m+1}(3,\k)$, and we may set $W'$ to be the representation generated by $(x_1-x_2)^{2m+1}L'$. Note that since $\frac{c_L}{c_K}QK$ is generated by $K$, $W'$ is just as good of a generating representation as $W$. Then $W'\subset Q_{m+1}(3,\k)_\std$, as desired.
\end{proof}

Finally, we are ready for the proof of Theorem \ref{thm:3.1}.

\begin{proof}[Proof of Theorem \ref{thm:3.1}]
    By (\ref{eq:3.1}), the Hilbert polynomial of $Q_m(3,\C)$ is $1+2t^{3m+1}+2t^{3m+2}+t^{6m+3}$, and a result of Ren and Xu \cite{Ren_2020} says that the Hilbert series of $Q_m(3,\k)$ is at least as big as the Hilbert series of $Q_m(3,\C)$ in all degrees. In particular, this implies that $Q_m(3,\k)_\std$ contains a generator of degree at most $3m+1$.
    
    We will first show that $Q_m(3,\k)_\std$ is freely generated by representations of degree $d$ and ${6m+3-d}$ for some $d$. We proceed by induction on $m$. The base case $m=0$ is a direct consequence of the fundamental theorem of symmetric polynomials, so we proceed directly to the inductive step.
    
    Assume that $Q_m(3,\k)_\std$ is freely generated by generators $V,W$ of degree $d$ and $e$ respectively where $d+e=6m+3$ and $d<e$. If $e-d>1$ then by Lemma \ref{lem:3.6} we may assume one of $V,W$ is in $Q_{m+1}(3,\k)_\std$. Let us assume it is $W$. Then note that $\prod_{i<j}(x_i-x_j)^2V$ and $W$ are both contained in $Q_{m+1}(3,\k)_\std$, and the sum of their degrees is $6m+3+6=6(m+1)+3$. If they are not generators then they must be generated by the same generator, as otherwise their respective generators would violate Lemma \ref{lem:3.4}. But this would imply that they are not independent over $\k[x_1,x_2,x_3]^{S_3}$, which cannot happen since $V$ and $W$ are independent over $\k[x_1,x_2,x_3]^{S_3}$. So they are generators, and by Lemma \ref{lem:3.5} $Q_{m+1}(3,\k)_\std$ is a free module generated by them. The same proof also works in the case when $V$ is in $Q_{m+1}(3,\k)$.
    
    Now, consider when $e-d=1$. There exists a generator $U$ of $Q_{m+1}(3,\k)_\std$ of degree at most $3(m+1)+1=d+3$ by the above. Note also that $U$ is generated by $V,W$ since it is contained in $Q_m(3,\k)_\std$ by Proposition \ref{prop:2.2}, so $\deg U\geq d$. We will proceed by casework on $\deg U$. If $\deg U=d$, then we must have $U=V$ since $V$ is the only standard representation in $Q_m(3,\k)$ of degree $d$. Then note that both $V$ and $\prod_{i<j}(x_i-x_j)^2W$ are in $Q_{m+1}(3,\k)_\std$, and they freely generate $Q_{m+1}(3,\k)_\std$ by the same argument as above. If $\deg U=d+1$ then we have $U=QV+W$ for some symmetric polynomial $Q$ of degree 1. But recall that all generators of the quasi-invariants lie in $\k[x_1-x_3,x_2-x_3]$, and the only symmetric polynomial of degree 1 is $P_1=x_1+x_2+x_3$, which does not lie in $\k[x_1-x_3,x_2-x_3]$. Thus we must have $Q=0$ and $U=W$. Then $W$ lies in $Q_{m+1}(3,\k)$, so the same proof as above holds. If $\deg U=d+2$, then since there are no symmetric polynomials of degree 1 in $\k[x_1-x_3,x_2-x_3]$ we must have $U=P_2V$. But then any $P\in U$ is divisible by $P_2$, which contradicts Corollary \ref{cor:3.3}. If $\deg U=d+3$, then by Lemma \ref{lem:3.4} there exists no other generators of $Q_{m+1}(3,\k)_\std$ of degree at most $d+3$. But then there must exist a generator of degree $d+4=3(m+1)+2$ since otherwise the Hilbert series of $Q_{m+1}(3,\C)$ would be larger than that of $Q_{m+1}(3,\k)$ in degree $3(m+1)+2$. Then by Lemma \ref{lem:3.5} $Q_{m+1}(3,\k)_\std$ is freely generated by these two generators.
    
    Thus $Q_m(3,\k)_\std$ is generated by representations of degree $d$ and $6m+3-d$. Note that one of $d,6m+3-d$ is at most $3m+1$, so we may assume $d\leq 3m+1$. Let $P\in Q_m(3,\k)_\std$ be of degree $d$. Then $(1-s_{i,j})P$ is also of degree $d$ and is nonzero for some $s_{i,j}$ since $P$ is not in the trivial representation. But it is divisible by $(x_i-x_j)^{2m+1}$, so it must be of degree at least $2m+1$. So we get $2m+1\leq d\leq 3m+1$. Note that $d$ only depends on $m,p$ and not on the ground field since quasi-invariants are defined by systems of linear equations with integer coefficients. Since std has dimension 2, we get that the Hilbert polynomial is $1+2t^d+2t^{6m+3-d}+t^{6m+3}$ so the Hilbert series is $\frac{1+2t^d+2t^{6m+3-d}+t^{6m+3}}{(1-t)(1-t^2)(1-t^3)}$.
\end{proof}
\begin{remark}
    In \cite{Ren_2020} Ren and Xu conjectured various properties of $Q_m(n,\k)$, those being that the largest degree term in its Hilbert polynomial is $t^{{n\choose{2}}(2m+1)}$, and that when $\char\k>2$ it is a free module of rank $n!$ and its Hilbert polynomial is palindromic. Theorem \ref{thm:3.1} proves all parts of this conjecture for $n=3$ and $\char\k>3$. We also remark that Proposition \ref{prop:2.6} shows that there exists a generator of $Q_m(n,\k)$ of degree ${n\choose{2}}(2m+1)$, though it is possible that there are relations of the same degree that cancel out with this generator in the Hilbert polynomial. We also note that even if the generator contributes a term to the Hilbert series
\end{remark}

\section{Conditions for the Hilbert series to differ from characteristic 0}\label{sec:4}

A question of interest that was studied in \cite{Ren_2020} was to find the values of $m,p$ such that $Q_m(n,\k)$ has a different Hilbert series than $Q_m(n,\C)$, where $\char\k=p$. The following theorem brings us closer to solving this problem for $n=3$ in the non-modular case.

\begin{theorem}\label{thm:4.1}
    Let $V,W$ be the generating representations of $Q_m(3,\C)$ and let $A\in V_{1,2}^-,B\in W_{1,2}^-$ where the coefficients of each of $A,B$ are coprime integers. Let 
    $$As_{1,2}B-Bs_{1,2}A=c\prod_{i<j}(x_i-x_j)^{2m+1}.$$
    Let $\mathrm{char}\,\k=p>3$. Then the Hilbert series of $Q_m(3,\k)$ is different from the Hilbert series of $Q_m(3,\C)$ if and only if $p$ divides $c$.
\end{theorem}
\begin{proof}
    First, assume that the Hilbert series is different in characteristic $p$. Then note that since the Hilbert polynomial over characteristic zero is $1+2t^{3m+1}+2t^{3m+2}+t^{6m+3}$ by (\ref{eq:3.1}), $Q_m(3,\k)_\std$ must be generated by representations of degree $d$ and $6m+3-d$ for $d<3m+1$ by Theorem \ref{thm:3.1}. Then note that $6m+3-d>3m+2$, so the images of $V$ and $W$ in characteristic $p$ are generated by the generator of degree $d$. But then $A,B$ are not independent over $\k[x_1,x_2,x_3]^{S_3}$, so we have $As_{1,2}B-Bs_{1,2}A=0$, and we must have $p|c$. 
    
    Now, assume that $p|c$. Then $A,B$ are not both generators of $Q_m(3,\k)$ by Lemma \ref{lem:3.4}, and they are nonzero since their coefficients are coprime. Let $\deg A=3m+1,\deg B=3m+2$. If $A$ is not a generator then it is generated by some representation of degree less than $3m+1$, so the Hilbert series in characteristic $p$ is different from that in characteristic zero since no such generator exists in characteristic zero. If $B$ is not a generator then it is generated by some representation of degree less than $3m+2$. But note that it cannot be generated by a representation of degree $3m+1$, as the only symmetric polynomial in degree 1 is $x_1+x_2+x_3$ and $B\in\k[x_1-x_3,x_2-x_3]$. So there must be a generator of degree less than $3m+1$, and the Hilbert series is different in characteristic $p$.
\end{proof}

Using the methods from the proof of Theorem \ref{thm:3.1}, we were able to calculate the generators $A,B$ as above recursively. Let us be more precise by calling the generators of $Q_m(3,\C)$ $A_m$ and $B_m$. Let $A_m=(x_1-x_2)^{2m+1}K_m,B_m=(x_1-x_2)^{2m+1}L_m$. Then we expressed $K_m,L_m$ as polynomials in $(x_1-x_2)^2$ with coefficients that are scalar multiples of 
$$(x_1-x_3)^r(x_2-x_3)^r,(x_1+x_2-2x_3)(x_1-x_3)^r(x_2-x_3)^r$$
for some $r$. Recalling that the generators of $\k[x_1-x_3,x_2-x_3]^{S_3}$ are $P_2,P_3$, note that $K_{m+1}$ is a linear combination of $P_3K_m$ and $P_2L_m$ and $L_{m+1}$ is a linear combination of $P_2^2K_m$ and $P_3L_m$. Thus by choosing appropriate linear combinations such that the constant terms as polynomials in $(x_1-x_2)^2$ vanish, we were able to construct the generators $K_{m+1},L_{m+1}$. We were then able to construct the coefficients $c$ in Theorem \ref{thm:4.1} and compute their prime factors to find the set of $p>3$ for which the Hilbert series of $Q_m(3,\C)$ differs from the Hilbert series of $Q_m(3,\k)$ when $\char \k=p$. 

In \cite{Ren_2020} Ren and Xu proved a sufficient condition, that if there exists integers $a,k$ that satisfy
\begin{equation}\label{eq:4.1}
    \frac{mn(n-2)+{n\choose{2}}}{n(n-2)k+{n\choose{2}}-1}\leq p^a\leq\frac{mn}{nk+1}\tag{4.1}
\end{equation}
then the Hilbert series of $Q_m(n,\mathbb{F}_p)$ is different from that of $Q_m(n,\C)$. They also conjectured that this condition is necessary. Using their programs, they confirmed this conjecture for $n=3,$\linebreak$m\leq 15, p<50$. Using the method described above, we were able to confirm the conjecture for a much larger range of values, namely $n=3,m\leq3000,3<p<\infty$. Thus we have strong evidence that the conjecture is true, at least in the case $n=3,p>3$.

Closer inspection of the inductive step in the proof of Theorem \ref{thm:3.1} also provides a way to compute the degrees of the generators in the cases where they are different in characteristic $p$. Let $V,W$ be the generating representations of $Q_m(3,\k)_\std$, and assume that they are of degree $3m+1$ and $3m+2$. Then from the proof of Theorem \ref{thm:3.1} we see that the generators of $Q_{m+1}(3,\k)_\std$ are of degree $3(m+1)+1$ and $3(m+1)+2$ unless either $V$ or $W$ is in $Q_{m+1}(3,\k)_\std$. For the sake of this argument it will not matter which representation it is, so let us say that it is $V$. Then let $k$ be the maximum positive integer such that $V\in Q_{m+k}(3,\k)$. For all $1\leq i\leq k$, we have that the minimal degree generator of $Q_{m+i}(3,\k)_\std$ is $V$, and for all $1\leq b\leq k$, the minimal degree generator of $Q_{m+k+b}(3,\k)$ is $\prod_{i<j}(x_1-x_j)^{2b}V$ of degree $3m+1+6b$. Note that when $b=k$ the lowest degree generator is of degree $3m+1+6k=3(m+2k)+1$, so the Hilbert series is equal to the Hilbert series in characteristic 0. 

Note that in \cite{Ren_2020} the authors explicitly computed elements of $Q_m(n,\k)_\std$ of degree less than $3m+1$ when $m,p$ satisfy (\ref{eq:4.1}) for some $a$. Assuming the condition (\ref{eq:4.1}) is necessary as conjectured by Ren and Xu, these elements coincide with the minimal degree generators that we described above. Thus a proof of the conjecture would also prove an explicit formula for the Hilbert series of $Q_m(3,\k)$ for all $m$ with $\char\k>3$. Namely, let $p,a,k$ satisfy \ref{eq:4.1} such that $a$ is as large as possible. Then if $p^a(2k+1)\geq 2m+1$, the Hilbert series of $Q_m(3,\k)$ would be
$$\frac{1+2t^{p^a(3k+1)}+2t^{6m+3-p^a(3k+1)}+t^{6m+3}}{(1-t)(1-t^2)(1-t^3)}$$
and if $p^a(2k+1)<2m+1$ then the Hilbert series would be
$$\frac{1+2t^{p^a(3k+2)}+2t^{6m+3-p^a(3k+2)}+t^{6m+3}}{(1-t)(1-t^2)(1-t^3)}.$$
Note that in the second case, the lowest degree generator of $Q_m(3,\k)_\std$ has degree \linebreak
${6m+3-p^a(3k+2)}$.

\subsection{The Opdam Shift Operator}

In this section we will describe some partial results towards proving the conjecture of Ren and Xu described above using the Opdam shift operators $O_m$ introduced in \cite{Opdam1989}.

\begin{definition}
    The Opdam shift operator $O_m$ is the differential operator (i.e. operator in the algebra $\Z\left[x_i,\partial_{x_i},\frac{1}{x_1-x_j}\right]$) whose action on symmetric polynomials is given by 
    \begin{equation}\label{eq:4.2}
        \prod_{i<j}\left(D_i(m)-D_j(m)\right)\prod_{i<j}(x_i-x_j)\tag{4.2}
    \end{equation}
    where
    $$D_i(k)=\partial_{x_i}-k\sum_{j\neq i}\frac{1}{x_i-x_j}(1-s_{ij})$$
    is the Dunkl operator.
\end{definition}

Note that even though $O_m$ is defined according to $\ref{eq:4.2}$, its action on polynomials is not given by $\ref{eq:4.2}$ since this is not a differential operator. Instead, to compute explicitly $O_m$ as a differential operator one must observe how $O_m$ acts on symmetric polynomials as a differential operator, and then take that differential operator as the definition of $O_m$.

\begin{example}\label{ex:4.3}
    In 3 variables, one can compute explicitly that
    $$O_m=(x_1-x_2)(x_1-x_3)(x_2-x_3)(\partial_{x_1}-\partial_{x_2})(\partial_{x_1}-\partial_{x_3})(\partial_{x_2}-\partial_{x_3})+6(1-2m)(1-3m)(2-3m)$$$$+\sum_{cyc}\Bigg[\left((-2+3m)(x_1-x_3)(x_2-x_3)+(1-2m)(x_1-x_2)^2\right)(\partial_{x_1}-\partial_{x_2})^2$$$$+\left((6-22m+20m^2)(x_1-x_2)+4m(-1+m)\left(\frac{(x_2-x_3)^2}{x_3-x_1}-\frac{(x_1-x_3)^2}{x_3-x_2}\right)\right)(\partial_{x_1}-\partial_{x_2})\Bigg].$$
\end{example}

The connection between Opdam shift operators and quasi-invariant polynomials was discovered in \cite{berest2003cherednik} and explained further in \cite{shift}. Explicitly, we have the following result:

\begin{theorem}[\cite{berest2003cherednik}, \cite{shift}]
    We have $O_m(Q_{m-1}(n,\C))\subset Q_m(n,\C)$.
\end{theorem}

Note also that the Opdam shift operators are $S_n$-invariant; this follows from the fact that the operator in Equation \ref{eq:4.2} is symmetric. In the case $n=3$, we can see this directly from Example $\ref{ex:4.3}$. Thus we can use the shift operators to inductively compute the generators of $Q_m(3,\C)_\std$, which we can use to describe the Hilbert series of $Q_m(3,\k)$ through Theorem \ref{thm:4.1}. 

Let $A_m, B_m$ be the degree $3m+1,3m+2$ generators of $Q_m(3,\C)$ whose coefficients have GCD equal to 1. Then let $P_2=\sum_{cyc}(x_i-x_j)^2,P_3=\prod_{cyc}(x_i+x_j-2x_k)$ be the degree 2 and 3 symmetric polynomials in $\C[x_1-x_2,x_2-x_3]$. Let 
$$O_{m+1}P_3A_m=a_mA_{m+1},\quad O_{m+1}P_3B_m=b_mB_{m+1},$$$$O_{m+1}P_2B_m=d_mA_{m+1},\quad O_{m+1}P_2^2A_m=e_mB_{m+1}.$$
Note that the Opdam shift operator is a degree 0 operator, and this combined with degree counting shows that the left hand sides of the above equations are indeed multiples of the respective generators $A_{m+1},B_{m+1}$. Then through computer calculations we found that for $m<100,p>3$ we have:

\begin{align}
    &\bullet v_p(a_m)\text{ is the number of }k>0\text{ with }m=1, 2\lfloor \frac{p^k}{3}\rfloor\text{ mod }p^k\tag{4.3}\label{4.3}\\
    &\bullet v_p(b_m)\text{ is the number of }k>0\text{ with }m=2, 2\lfloor\frac{p^k}{3}\rfloor-1\text{ mod } p^k\tag{4.4}\label{4.4}\\
    &\bullet v_p(d_m)\text{ is the number of }k>0\text{ with }p^k=5\text{ mod }6\text{ and }m=\frac{2p^k-4}{3}, \frac{2p^k-1}{3}\text{ mod }p^k\tag{4.5}\label{4.5}\\
    &\bullet v_p(e_m)\text{ is the number of }k>0\text{ with }p^k=1\text{ mod }6\text{ and }m=\frac{2p^k-5}{3}, \frac{2p^k-2}{3}\text{ mod }p^k\tag{4.6}\label{4.6}
\end{align}

Now, recall we have $A_m(sB_m)-(sA_m)B_m=c_m\prod(x_i-x_j)^{2m+1}$ for some $c_m$ where $s\in S_3$ is any transposition. Then from computer calculations we also have
\begin{equation}\label{4.7}
    a_mb_mc_{m+1}=(m-1)(m-2)(3m+1)(3m+2)c_m\tag{4.7}
\end{equation}
\begin{equation}\label{4.8}
    d_me_mc_{m+1}=(3m+1)(3m+2)c_m\tag{4.8}
\end{equation}
up to factors of $2,3$. Note that $a_mb_mc_{m+1},d_me_mc_{m+1}$ are expressable in terms of the Opdam shift operator as the leading coefficients of
$$(O_{m+1}P_3A_m)(O_{m+1}P_3(sB_m))-(O_{m+1}P_3(sA_m))(O_{m+1}P_3B_m),$$$$(O_{m+1}P_2B_m)(O_{m+1}P_4(sA_m))-(O_{m+1}P_2(sB_m))(O_{m+1}P_4A_m),$$
respectively, which we conjecture can be expressed in terms of $A_m(sB_m)-(sA_m)B_m$. Moreover, if one can prove either \ref{4.3}, \ref{4.4}, and \ref{4.7} or \ref{4.5}, \ref{4.6}, and \ref{4.8} for all $m$, then we will know the prime factorization of $c_m$ for all $m$ up to factors of $2,3$ by induction. A straightforward calculation shows that the prime factorizations we would obtain from these proofs coincides with the conjecture of Ren and Xu. Thus if we combine this with the rest of the results from this section, a proof of the equations \ref{4.3}, \ref{4.4}, and \ref{4.7} or \ref{4.5}, \ref{4.6}, and \ref{4.8} would complete a proof of the Hilbert series of $
Q_m(3,\k)$ for all $m,\k$ with $\char\k\neq 2,3$.
\section{$q$-deformed quasi-invariants}\label{sec:5}

For this section we will mainly work over the complex numbers $\C$.

\begin{definition}
    For a nonzero complex number $q$, a polynomial $P\in\C[x_1,\dots,x_n]$ is called \linebreak
    $q$-deformed $m$-quasi-invariant if for all transpositions $s_{i,j}\in S_n$, $(1-s_{i,j})P$ is divisible by \linebreak
    $\prod_{k=-m}^m(x_i-q^kx_j)$. We denote by $Q_{m,q}(n)$ the space of $q$-deformed $m$-quasi-invariants.
\end{definition}

Note that when $q=1$ we recover ordinary quasi-invariants, i.e.~$Q_{m,1}(n)=Q_m(n,\C)$. In \cite{BEF} it was proved that $Q_{m,q}(n)$ is a flat deformation of $Q_m(n,\C)$ when $q$ is a formal parameter, i.e.~the Hilbert series of $Q_{m,q}(n)$ over $\C(q)$ is equal to the Hilbert series of $Q_m(n,\C)$. This implies that if $q\in\C$, the Hilbert series of $Q_{m,q}(n)$ is equal to the Hilbert series of $Q_m(n,\C)$ for all but a countable set of values for $q$. The main goal of this section is to work toward understanding the values of $q$ for which the Hilbert series of $Q_{m,q}(n)$ differs. The following lemma uses the above to relate the various forms of quasi-invariants introduced in this paper.

\begin{lemma}
    \begin{enumerate}
        \item The Hilbert series of $Q_{m,q}(n)$ is at least as big as the Hilbert series of $Q_m(n,\C)$ in all degrees.\label{lem:5.21}
        \item Let $p$ be a prime, $q$ a $p$-th root of unity, and $\k$ a field of characteristic $p$. Then the Hilbert series of $Q_m(n,\k)$ is at least as big as the Hilbert series of $Q_{m,q}(n)$ in all degrees.\label{lem:5.22}
    \end{enumerate}
\end{lemma}
\begin{proof}
    Fix a degree $d$, and consider the space $Q_{m,\b{q}}(n)[d]$ over $\C(\b{q})$ where $\b{q}$ is a formal parameter. Note that the conditions for a polynomial to be in $Q_{m,\b{q}}(n)[d]$ can be described as a system of linear equations with coefficients in $\Z[\b{q},\b{q}^{-1}]$, so $Q_{m,\b{q}}(n)[d]$ can be described as the nullspace of a matrix with entries in $\Z[\b{q},\b{q}^{-1}]$. $Q_{m,q}(n)[d]$ is then the nullspace of the same matrix when we specialize $\b{q}=q$. Specializing a matrix can only increase the dimension of its nullspace, which proves \ref{lem:5.21}.
    
    Now, consider when $q$ is a $p$-th root of unity. Then note that $Q_{m,q}(n)[d]$ is the nullspace of a matrix with entries in $\Z[q,q^{-1}]$. If we specialize this matrix to a matrix in $\Z/p\Z[q,q^{-1}]$, the nullspace of this matrix becomes the space of polynomials in $\Z/p\Z[q,q^{-1}][x_1,\dots,x_n]$ such that $(1-s_{i,j})P$ is divisible by ${\prod_{k=-m}^m(x_i-q^kx_j)}$ for all $s_{i,j}$. But $q$ satisfies $q^p-1=0$, and in characteristic $p$ we have $q^p-1=(q-1)^p=0$, so $q=1$. Thus we obtain that this nullspace is exactly $Q_m(n,\k)[d]$, which proves \ref{lem:5.22}.
\end{proof}

Many basic properties of quasi-invariants also hold for $q$-deformed quasi-invariants with the obvious modifications and analogous proofs. Proposition \ref{prop:2.2}, Lemma \ref{lem:2.4}, and Proposition \ref{prop:2.5} directly translate from $Q_m(n,\k)$ to $Q_{m,q}(n)$. Analogously to Proposition 2.6, we have

\begin{proposition}\label{prop:5.3}
    As $\k[x_1,\dots,x_n]^{S_n}$-modules, we have
    \begin{enumerate}
        \item $Q_{m,q}(n)_\triv$ is freely generated by $1$.
        \item $Q_{m,q}(n)_\sgn$ is freely generated by $\prod_{i<j}\prod_{k=-m}^m(x_i-q^kx_j)$.
    \end{enumerate}
\end{proposition}

Thus by the same argument as in Remark \ref{rem:2.7}, we have

\begin{corollary}
    The Hilbert series of $Q_{m,q}(2)$ is $\frac{1+t^{2m+1}}{(1-t)(1-t^2)}$ for all $q$.
\end{corollary}

Similarly, Lemma \ref{lem:3.2}, Corollary \ref{cor:3.3}, Lemma \ref{lem:3.4}, and Lemma \ref{lem:3.5} all hold for $Q_{m,q}(3)_\std$ as well if all instances of $(x_i-x_j)^{2m+1}$ are replaced with $\prod_{k=-m}^m(x_i-q^kx_j)$ with the same proofs. We also have the following theorem analogous to Theorem \ref{thm:4.1}.

\begin{theorem}
    Let $V,W$ be the generating representations of $Q_{m,\b{q}}(3)$ and let $A\in V_{1,2}^-,B\in W_{1,2}^-$ where the coefficients of each of $A,B$ are coprime elements of $\C[\b{q},\b{q}^{-1}]$. Let 
    $$As_{1,2}B-Bs_{1,2}A=c\prod_{i<j}\prod_{k=-m}^m(x_i-q^kx_j).$$
    where $c\in\C[\b{q},\b{q}^{-1}]$. Then the Hilbert series of $Q_{m,q}(3)$ is different from the Hilbert series of $Q_m(3,\C)$ if and only if $q$ is a root of $c$.
\end{theorem}

\begin{proof}
    First, assume that the Hilbert series of $Q_{m,q}(3)$ is different from the Hilbert series of $Q_m(3,\C)$. Then one of the specializations of $A,B$ in $Q_{m,q}(3)$ is not a generating representation, so it is generated by a representation of lower degree. Without loss of generality, say that $A$ is generated by $C$. Then since $\deg A+\deg B=6m+3$, $\deg C+\deg B<6m+3$, so by Lemma \ref{lem:3.4} $B$ is also generated by $C$. But then $A,B$ are not independent over $\C[x_1,x_2,x_3]^{S_3}$, so $As_{1,2}B-Bs_{1,2}A=0$ and thus $c(q)=0$.
    
    Now, assume $q$ is a root of $c$. Then the specializations of $A,B$ in $Q_{m,q}(3)$ are not both generators by Lemma 3.4, and they are nonzero since their coefficients are coprime. Let us assume \linebreak
    $\deg A=3m+1,\deg B=3m+2$. If $A$ is not a generator then it is generated by some representation of degree less than $3m+1$, so the Hilbert series of $Q_{m,q}(3)$ is larger than the Hilbert series of $Q_m(3,\C)$ since no such generator exists in $Q_m(3,\C)$. If $B$ is not a generator then it is generated by some generator of degree less than $3m+2$. If this representation is not $A$ then we are done by the same argument as above. Note that if it is $A$, then up to a scalar we have $(x_1+x_2+x_3)A=B$ in $Q_{m,q}(3)$ since $x_1+x_2+x_3$ is the only symmetric polynomial of degree 1. We will show that this cannot happen, which will complete the proof.
    
    Let $M=Q_{m,\b{q}}(n)[3m+2]\cap\C[\b{q}][x_1,\dots,x_n]$ be a $\C[\b{q}]$-module. Then $N=M/(\b{q}-q)M$ is a module over $\C[\b{q}]/(\b{q}-q)\cong\C$, and it is easy to see that it is a submodule of $Q_{m,q}(n)[3m+2]$ that contains the specializations of $(x_1+x_2+x_3)A,B$ to $Q_{m,q}(n)$. 
    
    Let $e_1,\dots,e_k\in M$ be a basis of $Q_{m,\b{q}}(n)[3m+2]$ such that the images of $e_1,\dots,e_k$ in $N$ are nonzero. We will construct a basis of $k$ elements in $N$ using infinite descent on the sum of the degrees of $e_1,\dots,e_k$ in $\b{q}$. Let $\bar{e}_1,\dots,\bar{e}_k$ be the images of $e_1,\dots,e_k$ in $N$ and assume $\bar{e}_1,\dots,\bar{e}_k$ satisfy some linear relation $\sum_{i=1}^ka_i\bar{e}_i=0$ where $a_i\in\C$. Let $j$ be an index such that $e_j$ has the largest degree in $\b{q}$ out of all $e_i$ such that $a_i\neq 0$. Then $\sum_{i=1}^ka_ie_i$ is divisible by $\b{q}-q$ and has degree at most the degree of $e_j$. Since $e_1,\dots,e_k$ form a basis of $Q_{m,\b{q}}(n)[3m+2]$, it follows that  $e_1,\dots,\hat{e_j},\dots,e_k,\frac{1}{\b{q}-q}\sum_{i=1}^ka_ie_i$ is also a basis of $Q_{m,\b{q}}(n)[3m+2]$ that lies in $M$. $\frac{1}{\b{q}-q}\sum_{i=1}^ka_ie_i$ has a lower degree than $e_j$, so we have constructed a basis of $Q_{m,\b{q}}(n)[3m+2]$ that has a lower sum of degrees in $\b{q}$ than $e_1,\dots,e_k$. Since the sum of degrees must be nonnegative, this process must eventually terminate, after which there can be no relations between $\bar{e}_1,\dots,\bar{e}_k$. After this process, since $e_1,\dots,e_k$ form a basis of $M$, $\bar{e}_1,\dots,\bar{e}_k$ span $N$, so $N$ has dimension equal to $M$. Note that this implies $(x_1+x_2+x_3)A$ and $B$ cannot be linearly dependent in $N$, as they are linearly independent in $M$, so we have the desired result.
\end{proof}

However, the proof of Lemma \ref{lem:3.6} does not immediately generalize to the $q$-deformed case. This is because unlike $(x_i-x_j)^{2m+1}$, $\prod_{k=-m}^m(x_i-q^kx_j)$ is not contained in a subalgebra of $\C[x_1,x_2,x_3]$ of Krull dimension 2, and the 2 dimensional subalgebra is key to the proof. As such, we were unable to prove Lemma \ref{lem:3.6} and thus we do not know if Theorem \ref{thm:3.1} holds for the $q$-deformed case. However, based on our calculations of specific cases we conjecture that Theorem \ref{thm:3.1} indeed holds for $Q_{m,q}(3)$. 

Recall that in \cite{Ren_2020} Ren and Xu proved a sufficient condition for the Hilbert series of $Q_m(n,\k)$ to differ from the Hilbert series of $Q_m(n,\C)$. Using their methods, we proved a similar sufficient condition in the $q$-deformed case.

\begin{theorem}\label{thm:5.6}
    Let $m\geq0$, $n\geq 3$ be integers, and let $q$ be a primitive $p$-th root of unity, where $p$ is an integer such that
    $$\frac{mn(n-2)+{n\choose{2}}}{{n\choose{2}}-1}\leq p\leq mn.$$
    Then the Hilbert series of $Q_{m,q}(n)$ is different from the Hilbert series of $Q_m(n,\C)$.
\end{theorem}

Note the similarity between the inequality in this theorem and \ref{eq:4.1}. Namely, this inequality is obtained from \ref{eq:4.1} when we set $a=1,k=0$, but without the restriction that $p$ must be prime.

\begin{proof}
    In \cite{felder2003action} it was proved that the smallest positive degree term in the Hilbert series of $Q_m(n,\C)$ has degree $mn+1$. For the values of $q$ described above, we will exhibit a nonsymmetric $q$-deformed $m$-quasi-invariant of degree less than $mn+1$, which suffices for the proof.
    
    Let $q$ satisfy the condition in the theorem, and consider the polynomial 
    \begin{equation}\label{eq:5.1}
        P_{m,q}=\begin{cases}x_1^p-x_2^p&p\geq 2m+1\\(x_1^p-x_2^p)\prod_{i\neq j}\prod_{k=\frac{p+1}{2}}^m(x_i-q^kx_j)&p<2m+1, p\text{ odd}\\(x_1^p-x_2^p)\prod_{i<j}(x_i+x_j)\prod_{i\neq j}\prod_{k=\frac{p+2}{2}}^m(x_i-q^kx_j)&p<2m+1, p\text{ even}
        \end{cases}.\tag{5.1}
    \end{equation}
    Note that in both cases where $p<2m+1$, we have 
    $$\deg P_{m,q}=p+{n\choose{2}}(2m+1-p)\leq{n\choose{2}}(2m+1)-mn(n-2)-{n\choose{2}}=mn$$
    so $P_{m,q}$ always has degree less than $mn+1$. It suffices to show that $P_{m,q}\in Q_{m,q}(n)$. First, note that since $q$ is a primitive $p$-th root of unity, we have $$x_i^p-x_j^p=\prod_{k=-\frac{p-1}{2}}^{\frac{p-1}{2}}(x_i-q^kx_j)$$
    if $p$ is odd and
    $$x_i^p-x_j^p=\prod_{k=-\frac{p}{2}}^{\frac{p}{2}-1}(x_i-q^kx_j)$$
    if $p$ is even. If $p\geq 2m+1$, then this implies $x_i^p-x_j^p$ is divisible by $\prod_{k=-m}^m(x_i-q^kx_j)$. If $p<2m+1$ and $p$ is odd, then $P_{m,q}$ is divisible by 
    $$\prod_{k=-\frac{p-1}{2}}^{\frac{p-1}{2}}(x_1-q^kx_2)\prod_{k=\frac{p+1}{2}}^m(x_1-q^kx_2)(x_2-q^kx_1)$$$$=\pm q^l\prod_{k=-\frac{p-1}{2}}^{\frac{p-1}{2}}(x_1-q^kx_2)\prod_{k=\frac{p+1}{2}}^m(x_1-q^kx_2)(x_1-q^{-k}x_2)=\pm q^l\prod_{k=-m}^m(x_1-q^kx_2)$$
    for some integer $l$, and if $p$ is even $P_{m,q}$ is divisible by
    $$\prod_{k=-\frac{p}{2}}^{\frac{p}{2}-1}(x_1-q^kx_2)(x_1+x_2)\prod_{k=\frac{p+2}{2}}^m(x_1-q^kx_2)(x_2-q^kx_1)$$$$=\pm q^l\prod_{k=-\frac{p}{2}}^{\frac{p}{2}-1}(x_1-q^kx_2)(x_1-q^{\frac{p}{2}}x_2)\prod_{k=\frac{p+2}{2}}^m(x_1-q^kx_2)(x_1-q^{-k}x_2)=\pm q^l\prod_{k=-m}^m(x_1-q^kx_2).$$
    Now, we will show that $P_{m,q}\in Q_{m,q}(n)$ when $p\geq 2m+1$. Let $i,j>2$. Then we have
    $$(1-s_{1,2})(x_1^p-x_2^p)=2(x_1^p-x_2^p),\quad(1-s_{1,i})(x_1^p-x_2^p)=x_1^p-x_i^p,$$
    $$(1-s_{2,i})(x_1^p-x_2^p)=x_i^p-x_2^p,\quad(1-s_{i,j})(x_1^p-x_2^p)=0,$$
    and the desired result follows. Note that the products 
    $$\prod_{i\neq j}\prod_{k=\frac{p+1}{2}}^m(x_i-q^kx_j),\quad \prod_{i<j}(x_i+x_j)\prod_{i\neq j}\prod_{k=\frac{p+2}{2}}^m(x_i-q^kx_j)$$
    are symmetric, so the cases where $p<2m+1$ follow similarly.
\end{proof}

\begin{conjecture}\label{cong:5.7}
    The condition given in Theorem \ref{thm:5.6} is also necessary. That is, if the Hilbert series of $Q_{m,q}(n)$ is different from the Hilbert series of $Q_m(n,\C)$, then $q$ is a primitive $p$-th root of unity where $p$ is an integer that satisfies
    $$\frac{mn(n-2)+{n\choose{2}}}{{n\choose{2}}-1}\leq p\leq mn.$$
\end{conjecture}

In the case $n=3$, the Hilbert series of $Q_{m,q}(3)$ is different from the Hilbert series of $Q_m(3,\C)$ if and only if $Q_{m,q}(3)_\std$ contains a polynomial of degree $3m$. Note that by Proposition \ref{prop:5.3} there are no polynomials in $Q_{m,q}(3)_\sgn$ of degree $3m$. So using the characterization of $Q_{m,q}(3)$ given by Lemma \ref{lem:3.2}, the condition that a polynomial of degree $3m$ in $Q_{m,q}(3)_\std$ exists reduces to a polynomial equation in $q$. By directly computing these polynomials, we verified this conjecture for $n=3,m\leq 3$. Note also that this implies that for $n=3$ and a fixed $m$, there are only finitely many values of $q$ for which the Hilbert series of $Q_{m,q}(3)$ differs from the Hilbert series of $Q_m(3,\C)$. We conclude this section with a lemma about the lowest degree generator in the standard representation part of the $n=3$ case.

\begin{lemma}
    Let $n=3$ and $p,q$ satisfy the condition in Theorem \ref{thm:5.6}. Then the polynomial $P_{m,q}$ in \ref{eq:5.1} is the minimal degree generator of $Q_{m,q}(3)_\std$.
\end{lemma}
\begin{proof}
    Let us write $P_{m,q}=L\prod_{k=-m}^m(x_1-q^kx_2)$ for a polynomial $L$. Then assume for contradiction that there exists a generator of $Q_{m,q}(3)_\std$ of lower degree than $P_{m,q}$. By Lemma \ref{lem:3.5} there can only be one such generator, and by Lemma \ref{lem:3.2} this generator can be expressed as $K\prod_{k=-m}^m(x_1-q^kx_2)$ for a polynomial $K$. Since this is the only generator of degree at most the degree of $P_{m,q}$, $P_{m,q}$ must be a symmetric polynomial multiple of $K\prod_{k=-m}^m(x_1-q^kx_2)$, and thus $L$ must be a symmetric polynomial multiple of $K$. But it is easy to see directly from the definition of $P_{m,q}$ that $L$ is not divisible by any nontrivial symmetric polynomial, so we have a contradiction.
\end{proof}

\section*{Acknowledgements}
I would like to thank Calder Morton-Ferguson for mentoring this project and for providing me with various resources, advice, and ideas. I would like to thank Roman Bezrukavnikov for suggesting the original project and for useful discussions. I would also like to thank Pavel Etingof for suggesting a closely related but more tractable problem to pursue and for providing resources and valuable insight to help tackle the problem. I would also like to thank Michael Ren for useful discussions regarding his paper with Xu \cite{Ren_2020}. I would like to thank Ankur Moitra and David Jerison for overall advice about the project and research in general. Lastly, I would like to thank Slava Gerovitch and the MIT Math Department for running the Summer Program for Undergraduate Research (SPUR) and Undergraduate Research Opportunities Program (UROP), during which this research was conducted.

\printbibliography
\end{document}